\documentclass{article}
\usepackage{amsmath}
\usepackage{amssymb}
\usepackage{xcolor}
\usepackage{amsthm}
\usepackage[utf8]{inputenc}
\usepackage{hyperref}
\usepackage[english]{babel}

\newtheorem{theorem}{Theorem}[section]
\newtheorem{corollary}[theorem]{Corollary}
\newtheorem{conjecture}[theorem]{Conjecture}
\newtheorem{lemma}[theorem]{Lemma}
\newtheorem{proposition}[theorem]{Proposition}

\newcommand{\bbR}{\mathbb{R}}
\newcommand{\gauss}[2]{\genfrac{[}{]}{0pt}{}{#1}{#2}}
\newcommand{\Br}[1]{(#1)}

\title{The least Euclidean distortion constant of a distance-regular graph}
\author{

Sebastian M. Cioab\u{a}\footnotemark[1],

Himanshu Gupta\footnotemark[1],

Ferdinand Ihringer\footnotemark[2],\\

Hirotake Kurihara\footnotemark[3]}
\date{\today}
\begin{document}
\maketitle
\footnotetext[1]{Department of Mathematical Sciences, University of Delaware, Newark, DE 19716-2553, USA, {\tt \{cioaba,himanshu\}@udel.edu}. This research is partially supported by the NSF grant CIF-1815922.}
\footnotetext[2]{Dept.~of Mathematics:~Analysis, Logic and Discrete
Math., Ghent University, Belgium, {\tt ferdinand.ihringer@ugent.be}. The author is supported by a senior postdoctoral fellowship of the Research Foundation -- Flanders (FWO).}
\footnotetext[3]{Department of Applied Science, Yamaguchi University, Yoshida,  Yamaguchi, 753-8511, Japan, {\tt kurihara-hiro@yamaguchi-u.ac.jp} This research is partially supported by JSPS KAKENHI Grant Number JP20K03623.}

\begin{abstract}
In 2008, Vallentin made a conjecture involving the least distortion of an embedding of a distance-regular graph into Euclidean space.  Vallentin's conjecture implies that for a least distortion Euclidean embedding of a distance-regular graph of diameter $d$, the most contracted pairs of vertices are those at distance $d$. In this paper, we confirm Vallentin's conjecture for several families of distance-regular graphs. We also provide counterexamples to this conjecture, where the largest contraction occurs between pairs of vertices at distance $d-1$. We suggest three alternative conjectures and prove them for various families of distance-regular graphs and for distance-regular graphs of diameter 3.
\end{abstract}

\section{Introduction}\label{sec:intro}

Embeddings of graphs into Euclidean spaces have been well studied in mathematics and computer science. Linial, London, and Rabinovich \cite{LLR} investigated the distortion of such embeddings. Informally, the \textit{distortion} of an embedding  of a graph $G$ measures how much the combinatorial distance between two vertices in $G$ disagrees with their Euclidean distance in the embedding. This ties into a larger problem of embedding finite metric spaces in Hilbert spaces, as investigated by Bourgain \cite{Bourgain1985}.

In this paper, we continue an investigation by Vallentin \cite{FV} into the least Euclidean distortion constant of distance-regular graphs. We refer to \cite{LLR,FV} for more details on applications and related work. More formally, let $\| \cdot \|$ denote the norm of the standard inner product on $\bbR^n$,
that is $\| x \| = \sqrt{x_1^2 + \ldots + x_n^2}$. Let $(X,d)$ be a finite metric space
with $n$ elements. 
For an embedding $\rho: X \to \bbR^n$ define, as in \cite{LLR}, the following quantities:
\begin{itemize}
    \item expansion($\rho$) $:= \sup_{x,y \in X} \frac{\|\rho(x)-\rho(y)\|}{d(x,y)}$,
    \item contraction($\rho$) $:= \sup_{x,y \in X} \frac{d(x,y)}{\|\rho(x)-\rho(y)\|}$,
    \item distortion($\rho$) $:= \text{expansion}(\rho)\cdot \text{contraction}(\rho)$.
\end{itemize}
Let $c_2(X,d)$ denote the least distortion for which $(X,d)$ can be
embedded into $\bbR^n$. We say that an embedding of $(X,d)$ 
is \emph{optimal} if it has distortion $c_2(X,d)$. 
Any connected graph $G = (V,E)$ can be regarded as a finite metric space with point set $V$, where the distance between any two vertices equals the length of the shortest path between them. We denote the  least distortion of this finite metric space by $c_2(G)$.  In this paper, all embeddings are {\it faithful}, that is $\|\rho(x)-\rho(y)\|$ only depends on $d(x,y)$, not the choice of $x$ or $y$. Vallentin showed that this is no restriction in case of distance-regular graphs (see \cite[Lemma 3.2]{FV}).

Extending the work by Enflo for the hypercube \cite{Enflo1969} and
by Linial and Magen for cycles and products of cycles \cite{LM},
Vallentin \cite{FV} determined the least distortion constant
of Hamming graphs (including the hypercube), Johnson graphs, and all
strongly regular graphs. These are all classes of distance-regular graphs. A connected graph is {\em distance-regular} if it is regular of valency $k$, and if for any two vertices $x,y$ at distance $i$, there are precisely $c_i$ neighbors of $y$ at distance $i-1$ from $x$ and precisely $b_i$ neighbors at distance $i+1$ from $x$ \cite[p. 126]{BCN}. The set $\{ b_0, b_1, \ldots, b_{d-1}; c_1, c_2, \ldots, c_d \}$ is called {\it the intersection array} of the distance-regular graph. The adjacency matrix of a distance-regular graph $G$ has precisely $d+1$ distinct eigenvalues 
$\theta_0 > \theta_1 > \ldots > \theta_d$. Further, there exist univariate polynomials $v_i$ of degree $i$ such that $v_i(\theta_j)$ is an eigenvalue of the distance-$i$-graph of $G$ (and any eigenvalue of the distance-$i$-graph can be obtained that way). Denote $w_i(\theta_j) = \tfrac{v_i(\theta_j)}{v_i(\theta_0)}$. The sequence $(w_0(\theta_j),\ldots,w_d(\theta_j))$ is called {\em the standard sequence} of $G$ corresponding to the eigenvalue $\theta_j$ in \cite[p. 128]{BCN} and has the name {\em the $r$-cosine sequence} with respect to $\theta_j$ in \cite[p. 263]{G} (see Section \ref{sec:opt_embdedding}, \cite[Chapter 4]{BCN} and \cite[Chapter 13]{G} for a more details). In this more general framework, Vallentin proved the following result (see \cite[Theorem 2.4]{FV}).

\begin{theorem}[{Vallentin, 2008}] \label{thm:vallentin}
If $G$ is a distance-regular graph with diameter $d$, then
$$c_2(G)^2 \geq d^2 \min_{j \in \{1,\ldots,d\}} \left\{\frac{1-w_1(\theta_j)}{1-w_d(\theta_j)} \right\}.$$
\end{theorem}

In \cite{FV}, Vallentin also conjectured that the preceding result is tight.

\begin{conjecture}\label{FVconj}(Vallentin's Conjecture)
If $G$ is a distance-regular graph with diameter $d$, then 
$$c_2(G)^2 = d^2 \min_{j \in \{1,\ldots,d\}} \left\{\frac{1-w_1(\theta_j)}{1-w_d(\theta_j)} \right\}.$$
\end{conjecture}

In Section \ref{Counterexamples}, we disprove Conjecture \ref{FVconj} by presenting several counterexamples with diameter $4$ and larger. We also prove Conjecture \ref{FVconj} for several families of distance-regular graphs. Most notably, we prove the following result for distance-regular graphs with classical parameters (see Section \ref{sec:classical} for definitions).

\begin{theorem}\label{thm:beta_suff_large}
Let $G$ be a distance-regular graph with classical parameters $(d, b, \alpha, \beta)$ with $b\geq 1$. If $\beta$ is sufficiently large, then
\begin{align*}
     c_2(G)^2 = d^2 \frac{1-w_1(\theta_1)}{1-w_d(\theta_1)} = d^2 \frac{b^{d-1}}{\gauss{d}{1}_b}.
\end{align*}
\end{theorem}

In Section \ref{sec:classical} we verify Conjecture \ref{FVconj} for 
known families of distance-regular graphs with classical parameters.
We believe that Conjecture \ref{FVconj} is true with some minor modifications. Indeed, the proof 
of Theorem \ref{thm:vallentin} in \cite{FV} actually shows the following more general bound\footnote{Indeed, the maximum can range over all $r \in \{ 1, \ldots, d \}$, but we are not aware of any cases for which 
a value of $r \in \{ 1, \ldots, d-2\}$ yields a better bound.} (see Section \ref{Counterexamples} for details).

\begin{theorem}\label{thm:vallentin2}
If $G$ is a distance-regular graph with diameter $d$, then
\[c_2(G)^2 \geq \max_{r = d-1, d } \left\{ r^2 \min_{j \in \{1,\ldots,d\}} 
\left\{\frac{1-w_1(\theta_j)}{1-w_r(\theta_j)} \right\} \right\}.\]
\end{theorem}

We suggest the following revised version of Vallentin's conjecture.

\begin{conjecture}\label{Main_conjecture}(Vallentin's Conjecture, revised)
If $G$ is a distance-regular graph of diameter $d$, then
\begin{align*}
c_2(G)^2 =  \max_{r=d-1,d}\left\{r^2 \frac{1-w_1(\theta_1)}{1-w_r(\theta_1)} \right\}.
\end{align*}
The maximum occurs at $r=d$ unless $G$ is antipodal.
\end{conjecture}

We prove this conjecture for distance-regular graphs of diameter $3$ in Theorem \ref{diam3_Distortion}.
We also provide partial results towards this conjecture for distance-regular graphs of diameter $4$ in Section \ref{sec:small_d}.
In addition to Conjecture \ref{Main_conjecture},
we also propose the following two conjectures. 
\begin{conjecture}\label{Conjecture1}
Let $G$ be a distance-regular graph of diameter $d$ with eigenvalues $\theta_0 > \theta_1 > \ldots> \theta_d$. If $(w_0(\theta_1),w_1(\theta_1),\ldots,w_d(\theta_1))$ is the cosine sequence of $\theta_1$, then
 \[ \min_{1\leq r\leq d}\left\{ \frac{1-w_r(\theta_1)}{r^2}\right\} = \min_{r = d-1,d }\left\{ \frac{1-w_r(\theta_1)}{r^2}\right\}.\]
The minimum occurs at $r=d$ unless $G$ is antipodal.
\end{conjecture}

\begin{conjecture}\label{Conjecture2}
Let $G$ be a distance-regular graph with diameter $d$ and eigenvalues $\theta_0 > \theta_1 > \ldots> \theta_d$. For $1\leq j\leq d$, let $(w_0(\theta_j),w_1(\theta_j),\ldots,w_d(\theta_j))$ denote the cosine sequence of $\theta_j$. For any $r\in \{2,3,\ldots,d\}$, the following holds:
\[
    \max_{1\leq j\leq d}\left\{ \frac{1-w_r(\theta_j)}{1-w_1(\theta_j)}\right\} = \frac{1-w_r(\theta_1)}{1-w_1(\theta_1)}.
\]
\end{conjecture}

Note that Conjecture \ref{Conjecture1} and Conjecture \ref{Conjecture2} imply Conjecture \ref{Main_conjecture}. We have verified Conjecture \ref{Conjecture1} and Conjecture \ref{Conjecture2} (and therefore Conjecture \ref{Main_conjecture}) for all feasible intersection arrays given in Brouwer's list in \cite{BrouwerIA} (see \S\ref{subsec:IAs}).

\section{Optimal Embeddings} \label{sec:opt_embdedding}

Let $G$ be a distance-regular graph with $n$ vertices and diameter $d$. Consider the graph representation of $G$ on the eigenspace corresponding to an eigenvalue $\theta$ (see \cite[\S13]{G} for definitions). For us, usually $\theta = \theta_1$. 
Suppose $\theta$ has multiplicity $m$. Let $U_{\theta}$ be a $n \times m$ matrix with columns forming an orthonormal basis for the eigenspace associated with $\theta$. Denote by $u_{\theta}(i) \in \mathbb{R}^m$ the $i$-th row of $U_{\theta}$. Note that the Euclidean inner product $(u_{\theta}(i),u_{\theta}(i))$ does not depend on the vertex $i$ (see \cite[p. 262]{G}). If $i$ and $j$ are vertices in $G$ at distance $r$, then the $r$-th cosine $w_r(\theta)$ may be defined as: 
\begin{align*}
    w_r(\theta) = \frac{(u_{\theta}(i),u_{\theta}(j))}{(u_{\theta}(i),u_{\theta}(i))}.
\end{align*}
If we take the inner product of both sides of the equation 
\begin{align*}
    \theta u_{\theta}(i) = \sum_{j \sim i} u_{\theta}(j)
\end{align*}
with $u_{\theta}(\ell)$, where $\ell$ is a vertex with $d(\ell,i) = r$, then  we obtain the recurrence 
$$\theta w_r(\theta) = c_r w_{r-1}(\theta) + a_r w_r(\theta) + b_r w_{r+1}(\theta),$$
where $a_r=k-b_r-c_r$. It is easy to verify that  
$w_{r}(\theta_j) = \frac{v_{r}(\theta_j)}{v_r(\theta_0)}$
and we recover our definition of $w_r$ from Section \ref{sec:intro}.

Let us define the following embedding $\rho: V(G) \to \mathbb{R}^m$ as
\begin{equation}\label{eq:rhodef}
\rho: j \to \frac{1}{\sqrt{2(u_{\theta}(i),u_{\theta}(i))(1-w_1(\theta))}}u_{\theta}(j).\end{equation}
Lemma 13.3.1 of \cite{G} implies that $\rho$ is an injective function, thus, an embedding.
For two vertices $x$ and $y$ at distance $r$, we denote $S_r := \|\rho(x)-\rho(y)\|$. Using the above definition of $\rho$, we get that
$$
S_r=\|\rho(x)-\rho(y)\|=\frac{\sqrt{1-w_r(\theta)}}{\sqrt{1-w_1(\theta)}}.
$$

According to \cite[Claim 2.2]{LM} the most expanded pairs of vertices are adjacent. For the sake of completeness, we include the proof below for the embedding $\rho$. 
\begin{lemma}\label{simplelemma}
   Let $p$ and $q$ are two positive integers such that $pq\leq d$. Then 
   $$\frac{S_{pq}}{pq} \leq \frac{S_q}{q}.$$
   If $q=1$, then any pair of adjacent vertices are most expanded for $\rho$.
\end{lemma}
\begin{proof}
Consider two vertices $x$ and $y$ at distance $pq$. There exists vertices $x_0=x,x_1,\ldots,x_p=y$ such that $d(x_i,x_{i+1}) = q$ for all $i = 0,\ldots,p-1$. Using the triangle inequality, we obtain that
\[
    \frac{S_{pq}}{pq} = \frac{\|\rho(x)-\rho(y)\|}{d(x,y)} \leq \frac{\sum_{i=0}^{p-1}\|\rho(x_i)-\rho(x_{i+1})\|}{pq} = \frac{pS_q}{pq} = \frac{S_q}{q}. 
\]
Taking $q=1$ above proves the second assertion.\qedhere
\end{proof}

Thus, 
\begin{align*}
    \text{expansion}(\rho) = \sup_{x,y \in V(G)} \frac{\|\rho(x)-\rho(y)\|}{d(x,y)} = 
    \frac{S_1}{1} = 
    \frac{\sqrt{1-w_1(\theta)}}{\sqrt{1-w_1(\theta)}} = 1,
\end{align*}
and therefore, distortion$(\rho)$ $=$ contraction$(\rho)=\max_{1\leq r\leq d}\left\{\frac{r\sqrt{1-w_1(\theta_1)}}{\sqrt{1-w_r(\theta_1)}}\right\}$.

Conjecture \ref{Conjecture2} implies that the embedding $\rho$ in \eqref{eq:rhodef} with $\theta = \theta_1$ is an optimal embedding for any distance-regular graphs. 

We obtain some evidence for Conjecture \ref{Conjecture1}.
\begin{corollary}\label{HalfWay}
Let $G$ be a distance-regular graph of diameter $d$ with eigenvalues $\theta_0 > \theta_1 > \ldots> \theta_d$. If $(w_0(\theta_1),w_1(\theta_1),\ldots,w_d(\theta_1))$ is the cosine sequence of $\theta_1$, then
\[
    \min_{1\leq r\leq d}\left\{ \frac{1-w_r(\theta_1)}{r^2}\right\} = \min_{\lceil \frac{d+1}{2}\rceil \leq r \leq d}\left\{ \frac{1-w_r(\theta_1)}{r^2}\right\}.
\]
\end{corollary}
\begin{proof}
We apply Lemma \ref{simplelemma}.
For any $q$, such that $q\leq \lfloor\frac{d+1}{2}\rfloor$, there exist a positive integer $p$ such that $\lceil\frac{d+1}{2} \rceil\leq pq \leq d$.  
\end{proof}

\section{Antipodal Graphs and Counterexamples} \label{Counterexamples}

Linial, London and Rabinovich \cite[Corollary 3.5]{LLR} showed that the least Euclidean distortion of a graph $G=(V,E)$ is given by the equation:
\begin{equation*}
    c_2(G)^2=\max_{Q}\frac{\sum_{(x,y):q_{x,y}>0}d(x,y)^2q_{x,y}}{\sum_{(x,y):q_{x,y}<0}d(x,y)^2(-q_{x,y})},
\end{equation*}
where the maximum is taken over all positive semidefinite $V\times V$ matrix $Q=(q_{x,y})_{x,y\in V}$ having each row sum equal to zero. Given an embedding $\rho$ with minimal distortion $c_2(G)$, Linial and Magen \cite[Claim 1.4]{LM} proved that a matrix $Q$ attaining equality above must satisfy the following properties:
\begin{itemize}
    \item $q_{x,y}>0$ only for the most contracted pairs of vertices $(x,y)$, i.e., the pairs with $d(x,y)/\|\rho(x)-\rho(y)\|=$contraction$(\rho)$,
    \item $q_{x,y}<0$ only for the most expanded pairs $(x,y)$, i.e., the pairs with $\|\rho(x)-\rho(y)\|/d(x,y)=$expansion$(\rho)$,
    \item $q_{x,y}=0$ for all the other pairs of vertices.
\end{itemize}
Moreover, Linial and Magen \cite[Claim 2.2]{LM} showed that the most expanded pairs of vertices are always the pairs of adjacent vertices.

In his proof of Theorem \ref{thm:vallentin}, Vallentin used the above facts and assumed that the most contracted pairs of vertices are those at distance $d = \text{diam}(G)$, with the comment (see \cite[p. 6]{FV}):
\begin{center}
{\em so the lower bound can only be tight when the most contracted pairs are at distance $d$}.\end{center}

In this section, we show that this is not always the case and present several distance-regular graphs for which an optimal embedding has the property that the most contracted pairs of vertices are the pairs of vertices at distance $d-1$. All the examples we found are antipodal. In agreement with Conjecture \ref{Main_conjecture}, we are not aware of any distance-regular graphs for which the most contracted vertices occur at a distance other than $d-1$ or $d$.

\subsection{Proof of Theorem \ref{thm:vallentin2}}\label{subsec:vall2}

With these ideas in mind, let us now prove Theorem \ref{thm:vallentin2}. This can be done by using essentially the same argument as Vallentin's proof of Theorem \ref{thm:vallentin}. For $1\leq r\leq d$,
on page 6 of \cite{FV}, use the matrix $Q_\alpha = (k_1-\alpha k_r)A_0 - A_1 + \alpha A_r$
instead of $Q_\alpha = (k_1-\alpha k_d)A_0 - A_1 + \alpha A_d$. The same proof implies that
\[c_2(G)^2 \geq r^2 \min_{j \in \{1,\ldots,d\}} \left\{\frac{1-w_1(\theta_j)}{1-w_r(\theta_j)} \right\}.\]
Therefore, 
\begin{align*}
c_2(G)^2 &\geq \max_{r:1\leq r\leq d} r^2 \min_{j \in \{1,\ldots,d\}} \left\{\frac{1-w_1(\theta_j)}{1-w_r(\theta_j)} \right\}.
\end{align*}

\subsection{Antipodal Graphs}

A distance-regular graph is antipodal whenever $b_i = c_{d-i}$ for all $i = 0,1,\ldots,d$, 
except possibly $i = \lfloor d/2\rfloor$. If $G$ is antipodal, then by definition, 
the distance-$d$ graph $G_d$ is a disjoint union of cliques. These cliques are called the \emph{fibers} 
of $G$. These cliques are all the same size. We also say $G$ is an antipodal 
$r$-cover, where $r$ is the size of the cliques of $G_d$. When $r=2$, we say antipodal 
double-cover. The distinct eigenvalues of $K_n$ are $n-1$ and $-1$. We know 
sign($w_d(\theta_i)$) $=(-1)^i$, thus for an antipodal $r$-cover we have
\begin{align*}
    w_d(\theta_j) = \begin{cases}
    1, &\text{if } j \text{ is even,}\\
    -\frac{1}{r-1}, &\text{if } j \text{ is odd}.
    \end{cases}
\end{align*}

First, we show Conjecture \ref{Conjecture2} for $r = d-1$ and $d$ in case of antipodal graphs.
\begin{lemma}\label{lem:conj2_antipodal}
Let $G$ be an antipodal $r$-cover. Let $w_0(\theta_j),w_1(\theta_j),\ldots,w_d(\theta_j)$ denote the cosine sequence of $\theta_j$. For any $r = d-1, d$, we have 
\[
    \max_{1\leq j\leq d}\left\{ \frac{1-w_r(\theta_j)}{1-w_1(\theta_j)}\right\} = \frac{1-w_r(\theta_1)}{1-w_1(\theta_1)}.
\]
\end{lemma}
\begin{proof}
For $r=d$, we have $\min\{w_d(\theta_j)\} = w_d(\theta_1)$ and $\max\{w_1(\theta_j)\} = w_1(\theta_1)$.

For $r=d-1$, using the three-term recursion for $w_i$, we obtain
$\frac{w_{d-1}(\theta_j)}{w_d(\theta_j)} = \frac{\theta_j-a_d}{b_d}$.
Thus, we have
\begin{align*}
    w_{d-1}(\theta_j) = \begin{cases}
    \frac{\theta_j}{k}, &\text{if } j \text{ is even,}\\
   -\frac{1}{r-1}\cdot\frac{\theta_j}{k}, &\text{if } j \text{ is odd}.
    \end{cases}
\end{align*}
Thus, we have $\frac{1-w_{d-1}(\theta_j)}{1-w_1(\theta_j)}=1$ for $j$ even,
and $\frac{1-w_{d-1}(\theta_j)}{1-w_1(\theta_j)} = \frac{1+\frac{\theta_j}{k(r-1)}}{1-\frac{\theta_j}{k}}$
for $j$ odd.
\end{proof}

For $G$ an antipodal $r$-cover, then 
the expressions for $w_{d-1}(\theta_1)$ and $w_d(\theta_1)$ from the preceding proof imply that
\begin{align}
\frac{1-w_{d-1}(\theta_1)}{(d-1)^2} < \frac{1-w_d(\theta_1)}{d^2} \Longleftrightarrow \frac{\theta_1}{k} < \frac{d^2-2rd+r}{d^2}.
\label{eq:antipodal}
\end{align}

Notice that, if $d\leq 2r-1$, then $d^2-2rd+r <0$ but $\theta_1>0$.  Thus, we obtain the following important corollary.
\begin{corollary}%\label{antipodalcor1}
Let $G$ be an antipodal $r$-cover of diameter $d$ such that $d\geq 2r$.
Then $G$ is a counterexample of Conjecture \ref{FVconj} if and only if 
$\frac{\theta_1}{k} < \frac{d^2-2rd+r}{d^2}$.
\end{corollary}

In the next four subsections, we describe four counterexamples to Conjecture \ref{Conjecture1}.

\subsection{Counterexamples}

\subsubsection{Hadamard graphs}

For $\mu$ even, let $G$ denote a Hadamard graph with intersection array $\{2\mu, 2\mu-1, \mu, 1; 1,\mu, 2\mu-1,2\mu\}$, see \cite[\S1.8]{BCN}. We obtain the following $W_{ij} = w_j(\theta_i)$:
\begin{align*}
    W = \begin{bmatrix}
    1 &1 &1 &1 &1\\
    1 &\frac{1}{\sqrt{2\mu}} &0 &-\frac{1}{\sqrt{2\mu}} &-1\\
    1 &0 &-\frac{1}{2\mu -1} &0 &1\\
    1 &-\frac{1}{\sqrt{2\mu}} &0 &\frac{1}{\sqrt{2\mu}} &-1\\
    1 &-1 &1 &-1 &1
    \end{bmatrix}.
\end{align*}
According to Theorem \ref{thm:vallentin}, we have
\begin{align*}
    c_2(G)^2 &\geq 4^2\cdot \min\left\{\frac{1-\frac{1}{\sqrt{2\mu}}}{1+1},\frac{1-0}{1-1},\frac{1+\frac{1}{\sqrt{2\mu}}}{1+1},\frac{1+1}{1-1}\right\} = 8\left(\frac{\sqrt{2\mu}-1}{\sqrt{2\mu}}\right).
\end{align*}
Theorem \ref{thm:vallentin2} actually shows that (for $r=d-1=3$)
\[
    c_2(G)^2 \geq 3^2 \min\left\{\frac{1-\frac{1}{\sqrt{2\mu}}}{1+\frac{1}{\sqrt{2\mu}}},\frac{1-0}{1-0},\frac{1+\frac{1}{\sqrt{2\mu}}}{1-\frac{1}{\sqrt{2\mu}}},\frac{1+1}{1+1}\right\}= 9 \left(\frac{\sqrt{2\mu}-1}{\sqrt{2\mu}+1}\right).
\]

When $\mu \geq 34$, $9\left(\frac{\sqrt{2\mu}-1}{\sqrt{2\mu}+1}\right) > 8\left(\frac{\sqrt{2\mu}-1}{\sqrt{2\mu}}\right)$.
Thus, any Hadamard graph with $\mu \geq 34$ is a counterexample to Conjecture \ref{FVconj}.  The embedding in \eqref{eq:rhodef} shows that $$c_2(G)^2 = 9\left(\frac{\sqrt{2\mu}-1}{\sqrt{2\mu}+1}\right).$$

\subsubsection{Coset graph of the shortened binary Golay code} \label{subsec:golay1}
The \emph{coset graph of the shortened binary Golay code} $G$ is distance-regular of diameter $6$ with intersection array $\{22,21,20,3,2,1;1,2,3,20,21,\allowbreak 22\}$, see \cite[\S11.3H]{BCN}. 
Theorem \ref{thm:vallentin} only shows $c_2(G)^2 \geq \frac{126}{11} \sim 11.45$,
while Theorem \ref{thm:vallentin2} together with the embedding in Section \ref{sec:opt_embdedding}
shows that $c_2(G)^2 = \frac{35}{3} = 11.\overline{6}$.

\subsubsection{Double coset graph of truncated Golay code} \label{subsec:golay2}
The \emph{double coset graph of truncated Golay code} $G$ is distance-regular of diameter $7$ with intersection array $\{22,21,20,16,6,2,1;1,2,6,16, 20,\allowbreak21,22\}$, see \cite[\S11.3F]{BCN}. 
Theorem \ref{thm:vallentin} only shows $c_2(G)^2 \geq \frac{147}{11} \sim 13.36$, while Theorem \ref{thm:vallentin2} together with the embedding in Section \ref{sec:opt_embdedding} shows that $c_2(G)^2 = \frac{27}{2} = 13.5$.

\subsubsection{Double coset graph of binary Golay code} \label{subsec:golay3}
The \emph{double coset graph of binary Golay code} $G$ is distance-regular of diameter $7$ with intersection array $\{23,22,21,20,3,2,1;1,2,3,20,\allowbreak 21,22,23\}$, see \cite[\S11.3E]{BCN}.  
Theorem \ref{thm:vallentin} only shows $c_2(G)^2 \geq \frac{343}{23} \sim 14.91$,
while Theorem \ref{thm:vallentin2} together with the embedding in Section \ref{sec:opt_embdedding}
shows that $c_2(G)^2 = \frac{63}{4} = 15.75$.

\subsection{Feasible Intersection Arrays for Counterexamples} \label{subsec:IAs}

There are further intersection arrays which are feasible and could provide
counterexamples if corresponding distance-regular graphs were to be found.
The fact that we cannot find any counterexamples to Conjecture \ref{Main_conjecture}
even among feasible intersection arrays is in our opinion strong evidence
that it is correct. 
Feasible intersection arrays for distance-regular graphs of bipartite, antipodal
of diameter 4 are so plentiful that we will not list them here.

Intersection arrays with diameter at least 4 which are antipodal, but not bipartite,
on at most $2048$ vertices:

{\medskip \footnotesize \noindent
\setlength{\tabcolsep}{5pt}
\begin{tabular}{@{\,}crrr@{~~~}l}
$d$ & $v$ & IA & $c_2(G)^2$ & comments \\
\hline
4 & 1104 & $\{76,75,6,1;1,6,75,76\}$ & $\sim{}7.14773$ &  \\
4 & 1600 & $\{85,84,5,1;1,5,84,85\}$ & $\sim{}7.23867$ &  \\
4 & 1568 & $\{116,115,10,1;1,10,115,116\}$ & $\sim{}7.47073$ &  \\
4 & 1232 & $\{135,128,18,1;1,18,128,135\}$ & $7.2$ & $G_{1,2}$ SRG \\
4 & 1850 & $\{154,150,15,1;1,15,150,154\}$ & $7.5$ &  \\
4 & 2000 & $\{243,224,36,1;1,36,224,243\}$ & $7.2$ &  \\
6 & 2048 & $\{22,21,20,3,2,1;1,2,3,20,21,22\}$ & $\frac{35}{3}$ & \S\ref{subsec:golay1} \\
\end{tabular}\par}

\medskip 
Intersection arrays with diameter at least 5 which are antipodal and bipartite
on at most $2048$ vertices for diameter 5 and on at most $65536$ vertices otherwise.

{\medskip \footnotesize \noindent
\setlength{\tabcolsep}{5pt}
\renewcommand*{\arraystretch}{1.25}
\begin{tabular}{@{\,}crrr@{~~~}l}
$d$ & $v$ & IA & $c_2(G)^2$ & comments \\
\hline
5 & 704 & $\{26,25,24,2,1;1,2,24,25,26\}$ & $10$ &  \\
5 & 420 & $\{33,32,27,6,1;1,6,27,32,33\}$ & $\frac{64}{7}$ &  \\
5 & 704 & $\{36,35,32,4,1;1,4,32,35,36\}$ & $\frac{112}{11}$ &  \\
5 & 1408 & $\{37,36,35,2,1;1,2,35,36,37\}$ & $\frac{120}{11}$ &  \\
5 & 532 & $\{45,44,36,9,1;1,9,36,44,45\}$ & $\frac{176}{19}$ &  \\
5 & 784 & $\{46,45,40,6,1;1,6,40,45,46\}$ & $\frac{72}{7}$ &  \\
5 & 1276 & $\{49,48,45,4,1;1,4,45,48,49\}$ & $\frac{320}{29}$ &  \\
5 & 1300 & $\{55,54,50,5,1;1,5,50,54,55\}$ & $\frac{144}{13}$ &  \\
5 & 648 & $\{57,56,45,12,1;1,12,45,56,57\}$ & $\frac{28}{3}$ & $Q$-pol. \\
5 & 1104 & $\{76,75,64,12,1;1,12,64,75,76\}$ & $\frac{240}{23}$ &  \\
5 & 1600 & $\{85,84,75,10,1;1,10,75,84,85\}$ & $11.2$ &  \\
5 & 1334 & $\{96,95,80,16,1;1,16,80,95,96\}$ & $\frac{304}{29}$ &  \\
5 & 1568 & $\{116,115,96,20,1;1,20,96,115,116\}$ & $\frac{368}{35}$ & $Q$-pol. \\
7 & 4114 & $\{16,15,15,14,2,1,1;1,1,2,14,15,15,16\}$ & $\frac{180}{11}$ &  \\
7 & 2048 & $\{22,21,20,16,6,2,1;1,2,6,16,20,21,22\}$ & $13.5$ & \S\ref{subsec:golay2} \\
7 & 4096 & $\{23,22,21,20,3,2,1;1,2,3,20,21,22,23\}$ & $15.75$ & \S\ref{subsec:golay3} \\
7 & 19140 & $\{105,104,100,75,30,5,1;1,5,30,75,100,104,105\}$ & $\frac{468}{29}$ &  \\
\end{tabular}\par}

\section{Small Diameter} \label{sec:small_d}

\subsection{Strongly Regular Graphs}

Let $G$ be a strongly regular graph (SRG) with $v$ vertices and eigenvalues $k > r > s$.
In \cite[Theorem 2.5]{FV} it is shown that $c_2(G)^2 = \frac{4(v-k-1)(k-r)}{k(v-k+r)}$.
We found $c_2(G)^2 = 4(1 + \frac{1}{s})$ which is equivalent.\footnote{Note that in \cite[\S4.3]{FV} 
the eigenvalues of strongly regular graphs are stated incorrectly. This has no consequence, 
as these formulas are not used in \cite{FV}.}

Note that the proof of Theorem \ref{thm:vallentin} in 
\cite{FV} implies that the optimal embedding comes from the orthogonal projection onto the eigenspace of $\theta_1 = r$ (as is implied by Conjecture \ref{Conjecture2}).
We verified this explicitly.

\subsection{Diameter 3 or 4}

\begin{lemma}\label{Diam3or4}
Let $G$ be a distance-regular graph of diameter $3$ or $4$. Then Conjecture \ref{Conjecture1} holds true for $G$.
\end{lemma}

\begin{proof}
For $d = 3, 4$ we have $\lceil\frac{d+1}{2} \rceil = 2,3$ respectively. By using Corollary \ref{HalfWay}, we are done. 
\end{proof}

\begin{lemma}\label{Diam3_Conj2}
Let $G$ be a distance-regular graph of diameter $3$. Then Conjecture \ref{Conjecture2} holds true for $G$.
\end{lemma}

\begin{proof}
By using the three term recursion of $w_r(\theta)$ we obtain that
\begin{align*}
    \frac{1-w_2(\theta_j)}{1-w_1(\theta_j)} &= \frac{\theta_0 + \theta_j -a_1}{b_1},
\end{align*}
and
\begin{align*}
    \frac{1-w_3(\theta_j)}{1-w_1(\theta_j)} &= \frac{\theta_j^2 -(a_1+a_2-\theta_0)\theta_j + a_1a_2 -b_1c_2 - b_0c_1 - (a_1+a_2-\theta_0)\theta_0}{b_1b_2}.
\end{align*}
Since $\theta_1 > \theta_2 > \theta_3$, thus
$
   \max_{1\leq j \leq 3}\left\{\frac{1-w_2(\theta_j)}{1-w_1(\theta_j)}\right\} = \frac{1-w_2(\theta_1)}{1-w_1(\theta_1)}.
$
Moreover, to show that
$
   \max_{1\leq j \leq 3}\left\{\frac{1-w_3(\theta_j)}{1-w_1(\theta_j)}\right\} = \frac{1-w_3(\theta_1)}{1-w_1(\theta_1)}
$
is equivalent to show that $\theta_0+\theta_1+\theta_3 - (a_1 + a_2) \geq 0$. That is equivalent to show that $a_3 - \theta_2 \geq 0$, because $\theta_0 + \theta_1 + \theta_2 + \theta_3 = a_1  + a_2 + a_3$. Consider the matrix
\begin{align*}
    T = \begin{bmatrix}
    -c_1& b_1& 0\\
    c_1& k-b_1-c_2& b_2\\
    0& c_2& k-b_2-c_3\\ 
    \end{bmatrix}.
\end{align*}
The eigenvalue of $T$ are $\theta_1,\theta_2,\theta_3$ (see \cite[\textsection 4.1.B]{BCN}). As, $\begin{bmatrix}
-c_1& 0\\
0& k-b_2-c_3\\
\end{bmatrix}$ is a principal submatrix of $T$, so by the interlacing of the eigenvalues we have that $\theta_2 \leq k-b_2-c_3 = a_3-b_2$, or $\theta_2 \leq -c_1 = -1$. In either case, $a_3 - \theta_2 \geq 0$.
\end{proof}

\begin{theorem}\label{diam3_Distortion}
Let $G$ be a distance-regular graph of diameter 3 then Conjecture \ref{Main_conjecture} holds true for $G$. That is, 
\begin{align*}
    c_2(G)^2 &= \max\left\{4\frac{1-w_1(\theta_1)}{1-w_2(\theta_1)},9\frac{1-w_1(\theta_1)}{1-w_3(\theta_1)}\right\}\\ 
    &= \max\left\{\frac{4b_1}{\theta_0+\theta_1-a_1}, \frac{9b_1b_2}{(\theta_0+\theta_1-a_1)(\theta_0+\theta_1-a_2)-\theta_0\theta_1-b_1c_2-\theta_0}\right\}
\end{align*}
and further if $G$ is an antipodal $r$-cover, then 
\begin{align*}
    c_2(G)^2 = 9\frac{1-w_1(\theta_1)}{1-w_3(\theta_1)}
    &= \frac{9b_1b_2}{(\theta_0+\theta_1-a_1)(\theta_0+\theta_1-a_2)-\theta_0\theta_1-b_1c_2-\theta_0}.
\end{align*}
\end{theorem}
\begin{proof}
First assertion follows from Lemma \ref{Diam3or4} and Lemma \ref{Diam3_Conj2}. The second one follows since for any $r\geq 2$ we have $d^2-2rd+r = 9-5r < 0$ but $\theta_1 > 0$ and can use \eqref{eq:antipodal}. 
\end{proof}

\subsubsection{Taylor Graphs}

A distance-regular graph with intersection array 
$\{k,\mu,1;1,\mu,k\}$
is called a \emph{Taylor graph} \cite[\S1.5]{BCN}.
The eigenvalues of the Taylor graph are
$\theta_0 = k > \theta_1 = z_1> \theta_2 = -1> \theta_3 = z_2$, where $z_1$ and $z_2$ are the roots of $z^2-(k-1-2\mu)z-k = 0$, i.e.,
\begin{align*}
    z_1,z_2 &= \frac{(k-1-2\mu) \pm \sqrt{(k-1-2\mu)^2+4k}}{2}.
\end{align*}
For the matrix $W$, $W_{ij} = w_j(\theta_i)$, we obtain
 \begin{align*}
    W &= \begin{bmatrix}
    1& 1& 1& 1\\
    1& z_1/k& -z_1/k& -1\\
    1& -1/k& -1/k& 1\\
    1& z_2/k& -z_2/k& -1\\
    \end{bmatrix}.
\end{align*}
By using  Theorem \ref{diam3_Distortion}, we have that
\begin{align*}
    c_2(G)^2 = 3^2 \frac{k-z_1}{2k}. 
\end{align*}

\subsubsection{Generalized Polygons}

Let us consider the {\em point graph} $G$ of 
a generalized polygon of order $(s, t)$, see \cite[\S6.5]{BCN}
for details. In the following we generalize the results 
for classical generalized polygons with $s=t$ in \cite{Kobayashi2015}.
A generalized polygon with $s>1$ and $t>1$ whose point graphs is neither complete nor strongly regular,
is either a generalized hexagon with $s \leq t^3$ and $t \leq s^3$ or a generalized octagon 
with $s \leq t^2$ and $t \leq s^2$.
Note that in both cases Lemma \ref{Diam3or4} shows that Conjecture \ref{Conjecture1} is true for these graphs. In addition, a straightforward calculation implies that the minimum in Conjecture \ref{Conjecture1} is attained for $r=d$. 
Thus, to compute the least distortion we only need to prove Conjecture \ref{Conjecture2} for $r=d$. 

For $G$ the point graph of a generalized hexagon, 
we obtain the following $W_{ij} = w_j(\theta_i)$:
\[
    W = \begin{bmatrix}
    1 &1 &1 &1\\
    1 & \frac{s-1+\sqrt{st}}{s(t+1)} & \frac{-s+(s-1)\sqrt{st}}{s^2t(t+1)} & -1/st\sqrt{st} \\
    1 & \frac{s-1-\sqrt{st}}{s(t+1)} & \frac{-s-(s-1)\sqrt{st}}{s^2t(t+1)} & 1/st\sqrt{st}\\
    1 & -1/s & 1/s^2 & -1/s^3
    \end{bmatrix}.
\]

Note that $w_3(\theta_2) > 0$,
so we only need to compare $\frac{1-w_3(\theta_j)}{1-w_1(\theta_j)}$
for $j=1,3$ to show Conjecture \ref{Conjecture2} for $r=3$.
This is easily done, since $\frac{1-w_3(\theta_1)}{1-w_1(\theta_1)} > \frac{1-w_3(\theta_3)}{1-w_1(\theta_3)}$
$\Longleftrightarrow$ $(s^3+s^2)(t+1) + st^2(s^3t^2+1) + t^{3/2}(t(s^{7/2}-s^{3/2})+(s^{9/2}-s^{1/2})) >0.
$
Hence,
\begin{align*}
 c_2(G)^2 = 3^2 \frac{t\sqrt{st}}{(\sqrt{st}+1)(t+1)}.
\end{align*}

For $G$ the point graph of a generalized octagon,
we obtain
\[
    W = \begin{bmatrix}
    1 &1 &1 &1 & 1\\
    1 & \frac{s-1+\sqrt{2st}}{s(t+1)} & \frac{s(t-1)+(s-1)\sqrt{2st}}{s^2t(t+1)} & \frac{(s-1)st-s\sqrt{2st}}{s^3t^2(t+1)} & -1/s^2t^2 \\
    1 & (s-1)/s(t+1)  & -1/st  & -(s-1)/s^2t(t+1) & 1/s^2t^2\\
    1 & \frac{s-1-\sqrt{2st}}{s(t+1)}  & \frac{s(t-1)-(s-1)\sqrt{2st}}{s^2t(t+1)} & \frac{(s-1)st+s\sqrt{2st}}{s^3t^2(t+1)} & -1/s^2t^2\\
    1 & -1/s  & 1/s^2 & -1/s^3 & 1/s^4\\
    \end{bmatrix}.
\]

Note that $w_4(\theta_2) > 0$, $w_4(\theta_4)>0$, and $w_4(\theta_1) = w_4(\theta_3)$. Thus, Conjecture \ref{Conjecture2} holds true for $r=4$. Hence,
\begin{align*}
 c_2(G)^2 = 4^2 \frac{st^2(st-\sqrt{2st}+1)}{(t+1)(s^2t^2+1)}.
\end{align*}

\section{Graphs with Classical Parameters}\label{sec:classical}

Given an integer $b\neq 0,-1$ and two integers $m$ and $n$, define
\begin{align*}
    \gauss{n}{m} = \gauss{n}{m}_b = \begin{cases}
    0 &\text{if } m<0,\\
    \displaystyle{\binom{n}{m}} &\text{if }b=1,\\
    \displaystyle{\prod_{h=0}^{m-1}\frac{b^{n-h}-1}{b^{m-h}-1}} &\text{otherwise}.
    \end{cases}
\end{align*}

Let $d$ be a natural number. By the graphs \emph{with classical parameters} $(d, b, \allowbreak \alpha, \beta)$, we mean the distance-regular graphs with intersection numbers $b_i = (\gauss{d}{1}-\gauss{i}{1})(\beta - \alpha \gauss{i}{1})$ 
and $c_i = \gauss{i}{1}(1+\alpha\gauss{i-1}{1})$ ($0\leq i \leq d$) 
(see \cite[\textsection 6.1]{BCN}). It follows that $k = b_0= \beta\gauss{d}{1}$ 
and $a_i = k-b_i-c_i= \gauss{i}{1}(\beta - 1 + \alpha(\gauss{d}{1}-\gauss{i}{1}-\gauss{i-1}{1}))$. 
The eigenvalues of graphs with classical parameters are 
$\theta_i = \gauss{d-i}{1}(\beta - \alpha\gauss{i}{1}) - \gauss{i}{1} 
= b^{-i}b_i - \gauss{i}{1}$, $0\leq i \leq d$, (see \cite[Corollary 8.4.2]{BCN}). 
The following is a proof of Conjecture \ref{Conjecture1} for all 
distance-regular graphs of classical parameter with $b\geq 1$. 

\begin{theorem}\label{Thm_clas}
Let $G$ be a distance-regular graphs of classical parameter with $b\geq 1$ of diameter $d$. Let $w_0(\theta_1),w_1(\theta_1),\ldots,w_d(\theta_1)$ denote the cosine sequence of $\theta_1$. Then
\[
    \min_{1\leq r\leq d}\left\{ \frac{1-w_r(\theta_1)}{r^2}\right\} =  \frac{1-w_d(\theta_1)}{d^2}.
\]
\end{theorem}
\begin{proof}
Since $b\geq 1$, thus $\theta_0 > \theta_1 > \ldots > \theta_d$. Suppose $b=1$. Then the proof of \cite[Theorem 6.1.1]{BCN} implies that
\begin{align}\label{eqnb=1}
    1-w_r(\theta_1) =r\left(1-\frac{\theta_1}{k}\right).
\end{align}
Thus, 
\begin{align*}
\min_{1\leq r\leq d}\left\{ \frac{1-w_r(\theta_1)}{r^2}\right\} = \frac{1-w_d(\theta_1)}{d^2}.
\end{align*}
Suppose $b> 1$. According to \cite[Proposition 4.1.8 and Corollary 8.4.2]{BCN} we have 
\begin{equation*}
    w_r(\theta_1) = u \sigma_r + v = u \gauss{d-r}{1}_b +v,
\end{equation*}
for some numbers $u \neq 0$ and $v$. By using $w_0(\theta_1) = 1$ and $w_1(\theta_1) = \frac{\theta_1}{k}$ we solve for $u$ and $v$ and we obtained
\begin{align}\label{eqnb>1}
    {1-w_r} = \frac{b(1-b^{-r})}{(b-1)} \left(1-\frac{\theta_1}{k}\right).
\end{align}
Thus, the assertion follows.
\end{proof}

\begin{corollary}\label{Cor_class_b>1}
Let $G$ be a distance-regular graphs of classical parameter with $b\geq 1$. Then
\begin{align*}
    d^2 \min_{j \in \{1,\ldots,d\}}\left\{\frac{1-w_1(\theta_j )}{1-w_d(\theta_j)}\right\}\leq c_2(G)^2 \leq d^2\left\{\frac{1-w_1(\theta_1)}{1-w_d(\theta_1)}\right\} = d^2\frac{b^{d-1}}{\gauss{d}{1}_b}.
\end{align*}
\end{corollary}
\begin{proof}
The first inequality follows from Theorem \ref{thm:vallentin}. The second follows, since Theorem \ref{Thm_clas} implies that distortion$(\rho)^2$ $=$ $d^2\left\{\frac{1-w_1(\theta_1)}{1-w_d(\theta_1)}\right\}$ and $c_2(G)^2 \leq \text{distortion}(\rho)^2$. The last equality follows from (\ref{eqnb=1}) and (\ref{eqnb>1}) for $r=d$. 
\end{proof}

If Conjecture \ref{Conjecture2} holds true (for $r=d$) for graphs of classical parameters with $b\geq 1$ then the above inequality becomes equality. 
Now we can show Conjecture \ref{Conjecture2} for $\beta$ sufficiently large, that is
Theorem \ref{thm:beta_suff_large}.

\begin{proof}[Proof of Theorem \ref{thm:beta_suff_large}]
According to \cite[Theorem 4.5]{BCIM}, if $\beta$ is sufficiently large, then 
$\min_{1\leq j \leq d}\{w_d(\theta_j)\} = w_d(\theta_1)$.
We already know that $\max_{1\leq j \leq d}\{w_1(\theta_j)\} = w_1(\theta_1)$.
\end{proof}

\subsection{Known Examples}

In this subsection, we compute the least distortion constant 
of various graphs with classical parameters from Tables 6.1 and 6.2 in \cite{BCN}.
In each case, we provide a reference for the needed eigenvalues.
We only include explicit calculations for the arguably hardest case of
Hermitian forms graphs.
Our results suggest that the upper bound in Corollary \ref{Cor_class_b>1} is tight 
for all distance-regular regular graphs with classical parameters and $b \geq 1$.

We start with example with $b \geq 1$.
Note that we have Corollary \ref{Cor_class_b>1} in this case, which simplifies calculations.
The eigenvalues and therefore $c_2(G)$ only depend on the parameters $(d, b, \alpha, \beta)$.
If different graphs have the same parameters, then we refer to that entry.

\medskip
{\footnotesize \noindent
\setlength{\tabcolsep}{3pt}
% \begin{longtable}{@{\,}l@{~~~}ccccc@{~~}l}
\begin{tabular}{@{\,}l@{~~~}ccccc@{~~}l}
name & $d$ & $b$ & $\alpha$ & $\beta$ & $c_2(G)$ & ref \\
\hline
Hamming graph & $d$ & 1 & 0 & $q{-}1$ & $\sqrt{d}$ & \cite[Th. 2.5\Br{a}]{FV} \\
Johnson graph & $d$ & 1 & 1 & $n{-}d$ & $\sqrt{d}$ & \cite[Th. 2.5\Br{b}]{FV} \\
Halved cube & $d$ & 1 & 2 & $m$ & $\sqrt{d}$ & \cite[\S9.2D]{BCN} \\
Doob graph & $d$ & 1 & 0 & 3 & $\sqrt{d}$ & see Hamming, $q{=}4$\\
Grassmann graph & $d$ & $q$ & $q$ & $\gauss{n-d+1}{1}_q{-}1$ & $d \sqrt{q^{d-1} / \gauss{d}{1}_q}$ & \cite[\S5, Prop. 5.4\Br{iv}]{BCIM}\\
Twisted Grassmann 
& $d$ & $q$ & $q$ & $\gauss{n-d+1}{1}_q{-}1$ & $d \sqrt{q^{d-1} / \gauss{d}{1}_q}$ & see Grassmann\\
Bilinear forms graph & $d$ & $q$ & $q{-}1$ & $q^n{-}1$ & $d \sqrt{q^{d-1} / \gauss{d}{1}_q}$ & \cite[\S7, Prop. 7.3\Br{iii}]{BCIM}\\
\begin{tabular}{@{}c@{}}Dual polar graph \\ $e\in \{0,\frac{1}{2},1,\frac{3}{2},2\}$\end{tabular}
& $d$ & $q$ & 0 & $q^e$ & $d \sqrt{q^{d-1} / \gauss{d}{1}_q}$ & \cite[\S6]{BCIM} \\
Alternating forms graph & $d$ & $q^2$ & $q^2{-}1$ & $q^m{-}1$ & $d \sqrt{b^{d-1} / \gauss{d}{1}_b}$ & \cite[\S8, Prop. 8.3\Br{i}]{BCIM}\\
Quadratic forms graph & $d$ & $q^2$ & $q^2{-}1$ & $q^m{-}1$ & $d \sqrt{b^{d-1} / \gauss{d}{1}_b}$ & see alternating \\
\begin{tabular}{@{}l@{}}Half dual polar graph \\ $D_{n,n}(q)$ \end{tabular}
& $d$ & $q^2$ & $q^2+q$ & $\gauss{m+1}{1}_q{-}1$ & $d \sqrt{b^{d-1} / \gauss{d}{1}_b}$ & \cite[\S9.4C]{BCN} \\
\begin{tabular}{@{}l@{}}Dist. $1$-or-$2$ symplectic \\ dual polar graph \end{tabular}
& $d$ & $q^2$ & $q^2+q$ & $\gauss{m+1}{1}_q{-}1$ & $d \sqrt{b^{d-1} / \gauss{d}{1}_b}$ & see $D_{n,n}(q)$\\
Pseudo $D_m(q)$ graphs & $d$ & $q$ & 0 & 1 & $d \sqrt{b^{d-1} / \gauss{d}{1}_b}$ & dual polar, $e{=}0$ \\
Gosset graph $E_7(1)$ & 3 & 1 & 4 & 9 & $\sqrt{3}$ & \cite[\S3.11]{BCN}\\
\begin{tabular}{@{}l@{}}Exceptional Lie graph \\ $E_{7,7}(q)$ \end{tabular}
 & 3 & $q^4$ & $\gauss{5}{1}_q{-}1$ & $\gauss{10}{1}_q{-}1$ & $3 \sqrt{ b^2/\gauss{3}{1}_b}$ & \cite[\S10.7]{BCN}\\
Affine $E_6(q)$ graph & 3 & $q^4$ & $q^4{-}1$ & $q^9{-}1$ & $3 \sqrt{ b^2/\gauss{3}{1}_b}$ & \cite[\S10.8]{BCN} \\
% \end{longtable}\par}
\end{tabular}\par}

\medskip 
For $b \leq -1$, we have the following.

{\medskip \footnotesize \noindent
\setlength{\tabcolsep}{2.3pt}
\begin{tabular}{@{\,}l@{~~~}ccccc@{~~}l}
name & $d$ & $b$ & $\alpha$ & $\beta$ & $c_2(G)$ & ref \\
\hline
Witt graph $M_{24}$ & 3 & $-2$ & $-4$ & 10 & $\frac25 \sqrt{42}$ & \cite[\S11.4A]{BCN} \\
Witt graph $M_{23}$ & 3 & $-2$ & $-2$ & 5 & $2 \sqrt{\frac{21}{13}}$ & \cite[\S11.4B]{BCN} \\
\begin{tabular}{@{}l@{}}Extended ternary \\ Golay code graph\end{tabular}
& 3 & $-2$ & $-3$ & 8 & $\sqrt{\frac{33}{5}}$ & \cite[\S11.3A]{BCN} \\
Triality graph $^3D_{4,2}(q)$ & 3 & $-q$ & $\frac{q}{1{-}q}$ & $q^2{+}q$ & $3 \sqrt{\frac{q^5}{(q^3{+}1)(q^2{+}1)}}$ & \cite[\S10.7]{BCN} \\
\begin{tabular}{@{}l@{}}Unitary dual polar \\ graph $U(2d, \sqrt{q})$\end{tabular}
 & $d$ & $-\sqrt{q}$ & $\frac{q{+}\sqrt{q}}{1-\sqrt{q}}$ & $\frac{\sqrt{q}-(-\sqrt{q})^{d+1}}{1-\sqrt{q}}$ & $d \sqrt{q^{d-1} / \gauss{d}{1}_{q}}$ & 
 \begin{tabular}{@{}l@{}}dual polar \\ for $e=\frac12$\end{tabular}\\
\begin{tabular}{@{}c@{}}Hermitian forms graph \end{tabular}
& $d$ & $-q$ & $-q{-}1$ & $-(-q)^d{-}1$ & 
$d \sqrt{\frac{(b^{d}{+}b^{d{-}1}{+}b{+}1)b^{d{-}1}}{(b^d+b+1)\gauss{d}{1}_b}}$
& \S\ref{subsec:hermitian} \\
\end{tabular}\par}

\medskip

Note that $U(2d, \sqrt{q})$ is a dual polar graph with $e=\frac12$ which happens to have two sets of classical parameters, one with $b \geq 1$, see \cite[Corollary 6.2.2]{BCN}. 

\subsection{Hermitian forms graphs} \label{subsec:hermitian}

Hermitian forms graphs are graphs with classical parameters with $(d,b,\alpha,\beta)\allowbreak = (d,-q,-q{-}1,-(-q)^d{-}1)$ see \cite[\S9.5C]{BCN}. We know that
\begin{align}\label{HermEigenmat}
v_j(\theta_i) = (-1)^j \sum_{h=0}^{j} (-q)^{\binom{j-h}{2}+hd}\gauss{d-h}{d-j}_b \gauss{d-i}{h}_b
\end{align}
(see \cite[\S9]{BCIM}, \cite{Schmidt}, and \cite{Stanton}). Here the Gaussian coefficients have base $b = -q$. The eigenvalues are $\theta_i = ((-q)^{2d-i}-1)/(q+1)$ for $i=0,\ldots,d$. Note that, the second largest eigenvalue here is $\theta_2$ rather than $\theta_1$. Thus, we prove conjectures \ref{Conjecture1} and \ref{Conjecture2} with respect to $\theta_2$. Since Hermitian forms graphs are self-dual, thus $w_r(\theta_i) = w_i(\theta_r)$ for all $i = 0,\ldots,d$ and $r=0,\ldots,d$. We prove Conjecture \ref{Conjecture1} and Conjecture \ref{Conjecture2} below.

\begin{proposition}
Let $G$ be a Hermitian forms graph. Then
\begin{align*}
    \min_{1\leq r\leq d}\left\{\frac{1-w_r(\theta_2)}{r^2}\right\}
     = \frac{1-w_d(\theta_2)}{d^2}. 
\end{align*}
\end{proposition}
\begin{proof}
We know $\displaystyle{w_r(\theta_2) = w_2(\theta_r) = \frac{v_2(\theta_r)}{v_2(\theta_0)}}$. By using \eqref{HermEigenmat}, we get that
\begin{align*}
    \frac{1-w_r(\theta_2)}{r^2} = \frac{q^{2d-1}(1-q^{-2r})+(q-1)((-1)^r q^{-r}-1)}{(q^{2d-1}-q)(1-q^{-2d})r^2}.
\end{align*}
The assertion follows.
\end{proof}

\begin{proposition}
Let $G$ be a Hermitian forms graph. Then
\begin{align*}
    \max_{1\leq j\leq d} \left\{\frac{1-w_d(\theta_j)}{1-w_1(\theta_j)}\right\} = \frac{1-w_d(\theta_2)}{1-w_1(\theta_2)}.
\end{align*}
\end{proposition}
\begin{proof}
We know $\displaystyle{w_d(\theta_j) = w_j(\theta_d) = \frac{v_j(\theta_d)}{v_j(\theta_0)} = \frac{v_j(\theta_d)}{k_j}}$. By using (\ref{HermEigenmat}) to compute $v_j(\theta_d)$ and \cite[Corollary 8.4.4]{BCN} to compute $k_j$, we get that 
\begin{align*}
    w_d(\theta_j) = \left(\prod_{h=0}^{j-1}((-q)^{d-h}+1)\right)^{-1}.
\end{align*}
Recall that
\begin{align*}
    w_1(\theta_j) = \frac{\theta_j}{\theta_0} = \frac{(-q)^{2d-j}-1}{(-q)^{2d}-1}
\end{align*}
and that $\displaystyle{\max_{1\leq j \leq d} \{w_1(\theta_j)\} = w_1(\theta_2)}$. Note that we have 
\begin{align*}
    \frac{|w_d(\theta_j)|}{|w_d(\theta_{j+1})|} = |(-q)^{d-j}+1| \geq 1
\end{align*}
as well as
\begin{align*}
    w_d(\theta_1) &= \frac{1}{(-q)^d+1}, \text{ and }\\
    w_d(\theta_2) &= \frac{1}{((-q)^{d}+1)((-q)^{d-1}+1)}.
\end{align*}
Thus, if $d$ is even then $\displaystyle{\min_{1\leq j \leq d}\{w_d(\theta_j)\} = w_d(\theta_2)}$. Hence,
\begin{align*}
    \max_{1\leq j\leq d} \left\{\frac{1-w_d(\theta_j)}{1-w_1(\theta_j)}\right\} = \frac{1-w_d(\theta_2)}{1-w_1(\theta_2)}.
\end{align*}
On the other hand, if $d$ is odd then 
\begin{align*}
\displaystyle{\min_{2\leq j \leq d}\{w_d(\theta_j)\} = w_d(\theta_2)},
\end{align*}
but $\min\{w_d(\theta_1),w_d(\theta_2)\} = w_d(\theta_1)$.  
Hence,
\begin{align*}
    \max_{1\leq j\leq d} \left\{\frac{1-w_d(\theta_j)}{1-w_1(\theta_j)}\right\} &= \max\left\{\frac{1-w_d(\theta_1)}{1-w_1(\theta_1)},\frac{1-w_d(\theta_2)}{1-w_1(\theta_2)}\right\}\\
    &= \max\left\{\frac{q^d+1}{q^{d-1}(q+1)},\frac{(q^d+1)(q^d+q-1)}{q^{d-1}(q+1)(q^d+q-1-q^{d-1})}\right\}\\
    &= \frac{1-w_d(\theta_2)}{1-w_1(\theta_2)}. \qedhere
\end{align*}
\end{proof}
Thus, by using the above two propositions we have
\begin{align*}
    c_2(G) &= d\sqrt{\frac{1-w_1(\theta_2)}{1-w_d(\theta_2)}} = d \sqrt{\frac{(b^{d}+b+1+b^{d-1})b^{d-1}}{(b^d+b+1)\gauss{d}{1}_b}}.
\end{align*}

\section{Odd Graphs}

For natural numbers $n\geq 2d$, the Johnson graph $J(n,d)$ has as its vertices the $d$-subsets of a given set with $n$ elements, where two $d$-subsets are adjacent if and only they meet in a $(d-1)$-set. For $0\leq j\leq d$, the distance-$j$ graph $J(n,d,j)$ of the Johnson graph $J(n,d)$ has the same vertex set as $J(n,d)$ and two $d$-subsets are adjacent in $J(n,d,j)$ if and only if they are at distance $j$ in $J(n,d)$ which is equivalent to their intersection having size $d-j$. The eigenvalues of the Johnson graph $J(n,d,j)$ are given by the Eberlein polynomials $E_j(i)$ for $0\leq i,j\leq d$, where
\begin{equation*}
    E_j(i)=\sum_{h=0}^{j}(-1)^h\binom{i}{h}\binom{d-i}{j-h}\binom{n-d-i}{j-h}.
\end{equation*}
See also \cite[\S3]{BCIM} for other formulas of these eigenvalues.

The Odd graph $O_{d+1}$ is the distance $d$ graph of the Johnson graph $J(2d+1,d,1)$. Its vertices are the $d$-subsets of $[2d+1]$ and two vertices  are adjacent if and only if they are disjoint. If $x$ and $y$ are two vertices at distance $r$ in $O_{d+1}$, then
\begin{align*}
    |x\cap y| = \begin{cases}
    d-\frac{r}{2}, &\text{if } r \text{ is even},\\
    d-\frac{2d-r+1}{2}, &\text{if } r \text{ is odd}. 
    \end{cases}
\end{align*}
We use the Eberlein polynomials above to compute cosine sequences of the Odd graph. The eigenvalues of $O_{d+1}$ are $\theta_i = (-1)^i(d+1-i)$, for $0\leq i\leq d$. Note that the second largest eigenvalue here is $\theta_2$ rather than $\theta_1$. Thus, we prove Conjecture \ref{Conjecture1} and Conjecture \ref{Conjecture2} with respect to $\theta_2$.

\begin{proposition}
For the Odd graph $O_{d+1}$, 
\begin{align*}
    \min_{1\leq r\leq d}\left\{\frac{1-w_r(\theta_2)}{r^2}\right\} = \frac{1-w_d(\theta_2)}{d^2}.
\end{align*}
\end{proposition}
\begin{proof} 
For $j\in\{1,\ldots,d\}$, we have
\begin{align*}
    1-\frac{E_j(2)}{E_j(0)} = \frac{4jd^2-4j^2d+2j(j-1)}{d(d^2-1)}.
\end{align*}
Thus, for $j_1 = r/2$ and $j_2 = (2d-r+1)/2$, 
\begin{align*}
 \frac{1-w_r(\theta_2)}{r^2} &= \frac{1-\frac{E_{j_1}(2)}{E_{j_1}(0)}}{r^2} = \frac{\frac{2(2d^2-1)}{r}-(2d-1)}{2d(d^2-1)}
\intertext{for $r$ even, and}
\frac{1-w_r(\theta_2)}{r^2} &= \frac{1-\frac{E_{j_2}(2)}{E_{j_2}(0)}}{r^2} = \frac{\frac{4d^2}{r}-\frac{2d+1}{r^2}-(2d-1)}{2d(d^2-1)}
\end{align*}
for $r$ odd.
 From these formulas the claim follows.
\end{proof}

The distance-$d$ graph of $O_{d+1}$ is the distance-$\lceil \frac{d}{2} \rceil$ graph of Johnson graph $J(n,d)$, i.e., it is $J(2d+1,d,\lceil \frac{d}{2} \rceil)$, whose eigenvalues are
\begin{align*}
    E_{\lceil \frac{d}{2}\rceil} (i) = \sum_{h=0}^{i}(-1)^{i-h}\binom{i}{h}\binom{d-h}{\lceil \frac{d}{2} \rceil}\binom{d-i+h+1}{\lfloor \frac{d}{2} \rfloor+1}.
\end{align*}

\begin{lemma}
The smallest eigenvalue of $J(2d+1,2d,\lceil \frac{d}{2} \rceil )$ is $E_{\lceil \frac{d}{2} \rceil}(2)$.
\end{lemma}
\begin{proof}
Suppose $d=2\ell+1$, then, by \cite[Proposition 3.4]{BCIM}, $E_{\lceil\frac{d}{2}\rceil}(2) = E_{\lceil\frac{d}{2}\rceil}(1)$. The latter is the smallest eigenvalue according to \cite[Theorem 3.10]{BCIM}. Now, suppose $d=2\ell$. For $\ell\leq 6$, we verified the statement by computer. Thus, we assume that $\ell\geq 7$. The multiplicity of each eigenvalue $E_\ell(i)$ is $m_i = \binom{2d+1}{i}-\binom{2d+1}{i-1}$. Hence,
\begin{align*}
    \sum_{i=3}^d m_i E_\ell(i)^2 &= v\cdot E_\ell(0) - \sum_{i=0}^2m_iE_\ell(i)^2\\
    &= \frac{(4\ell+1)!}{\ell!^4(\ell+1)}-\frac{(12\ell^3+7\ell^2-3\ell-1){(2\ell)!}^4}{(2\ell-1)(\ell+1)^2 \ell!^8}.
\end{align*}
For all $i\in \{3,\ldots,d\}$, $m_i \geq m_3 = \frac{8\ell(4\ell+1)(\ell-1)}{3}$, thus we have
\begin{align*}
    |E_\ell(i)| \leq \sqrt{\frac{3}{8\ell(4\ell+1)(\ell-1)}\left(\frac{(4\ell+1)!}{\ell!^4(\ell+1)}-\frac{(12\ell^3+7\ell^2-3\ell-1)(2\ell)!^4}{(2\ell-1)(\ell+1)^2\ell!^8}\right)}.
\end{align*}
Hence, 
\begin{align*}
    \frac{|E_\ell(i)|}{|E_\ell(2)|} \leq \sqrt{\frac{3}{8\ell^3(\ell{-}1)}\left(\frac{(4\ell)!\ell!^4}{(2\ell)!^4}(\ell{+}1)(2\ell{-}1)^2-\frac{(12\ell^3{+}7\ell^2{-}3\ell{-}1)(2\ell{-}1)}{(4\ell{+}1)}\right)}.
\end{align*}
By induction on $\ell\geq 7$, we verify that $\frac{(4\ell)!\ell!^4}{(2\ell)!^4} \leq \frac{\ell}{2}$.
This implies that
\begin{align*}
     \frac{|E_\ell(i)|}{|E_\ell(2)|} \leq \sqrt{\frac{48\ell^5-60\ell^4-42\ell^3+51\ell^2-3}{64\ell^5-48\ell^4-16\ell^3}} < 1.
\end{align*}
Since $E_\ell(1)>0$ and $E_\ell(2)<0$, 
$\min_{1\leq i\leq d}\{E_\ell(i)\} = E_\ell(2)$. 
\end{proof}

Hence, we obtain the following corollary.
\begin{corollary}
For the Odd graph $O_{d+1}$, 
\begin{align*}
   \max_{1\leq j \leq d}\left\{\frac{1-w_d(\theta_j)}{1-w_1(\theta_j)}\right\}  = \frac{1-w_d(\theta_2)}{1-w_1(\theta_2)}.
\end{align*} 
\end{corollary} 

Thus,
\begin{align*}
   c_2(O_{d+1}) = d \sqrt{\frac{1-w_1(\theta_2)}{1-w_d(\theta_2)}} = \begin{cases}
   2\ell\sqrt{\frac{4\ell-2}{4\ell^2+\ell-1}}, &\text{if }d=2\ell,\\
   (2\ell+1)\sqrt{\frac{4\ell+2}{4\ell^2+7\ell+3}}, &\text{if }d=2\ell+1.
   \end{cases}
\end{align*} 

\section{Final Remarks}

In this paper, we disproved a conjecture of Vallentin on the least Euclidean distortion of distance-regular graphs. We proposed a revised conjecture related to the least Euclidean distortion and two related conjectures involving the cosine sequences of distance-regular graphs. We proved our conjectures for several families of graphs and presented computational arguments in their favor. Instead of considering the least distortion with respect to Euclidean norm, one may ask what happens for other norms. The least distortion $c_p(G)$ with respect to the $p$-norm for a graph $G$ is a natural generalization of $c_2(G)$ and was investigated by Jolissaint and Valette in \cite{Jolissaint2014}. In particular, they determine $c_p(G)$ for the hypercube. It would be interesting to investigate $c_p(G)$ for other distance-regular graphs as well.

\paragraph*{Acknowledgment} We thank Chris Godsil, Krystal Guo, Jack Koolen and Bill Martin for valuable comments.

%-----References-----

\end{document}